\documentclass[12pt,onesided]{article}

\usepackage[leqno]{amsmath}
\usepackage{amsthm}
\usepackage{amssymb}
\usepackage{amsfonts}
\usepackage{array}
\usepackage{hyperref}

\usepackage{xy}
\input xy
\xyoption{all}

\newcommand{\F}{\mathcal{F}}

\newcommand{\f}{\mathcal{F}}
\newcommand{\p}{\mathbb{P}}

\newtheorem{theorem}{THEOREM}

\newtheorem{proposition}{PROPOSITION}

\newtheorem{definition}{DEFINITION}

\def\bi{\begin{itemize}}
\def\ei{\end{itemize}}

\begin{document}

\title{Minimal elements of stopping time $\sigma$-algebras}
\author{Tom Fischer\thanks{Institute of Mathematics, University of Wuerzburg, 
Campus Hubland Nord, Emil-Fischer-Strasse 30, 97074 Wuerzburg, Germany.
Tel.: +49 931 3188911.
E-mail: {\tt tom.fischer@uni-wuerzburg.de}.
}\\
{University of Wuerzburg}} 
\date{This version: \today\\First version: December 13, 2011}
\maketitle

\begin{abstract}
We show how minimal elements of a stopping time $\sigma$-algebra can be expressed in terms
of the minimal elements of the $\sigma$-algebra of the underlying filtration.
This facilitates an intuitive interpretation of stopping time $\sigma$-algebras.
An example is provided.
\end{abstract}




We denote the time axis by $\mathbb{T}$, where $\mathbb{T}\subset\mathbb{R}$.

\begin{definition}
\label{def1}
Let $\tau$ be a stopping time on a filtered probability space $(\Omega, \f_\infty, (\f_t)_{t\in\mathbb{T}}, \p)$. 
We define the stopping time $\sigma$-algebra w.r.t.~$\tau$ as
\begin{equation}
\label{stsa}
\f_{\tau} = \{F\in\f_{\infty}: F\cap\{\tau\leq t\} \in\f_t \,\text { for all } t\in\mathbb{T}\} .
\end{equation}
\end{definition}
It is well known and straightforward to show that $\f_{\tau}$ indeed is a $\sigma$-algebra.
Note again that $\f_t \subset \f_{\infty}$ is assumed for $t\in\mathbb{T}$.

\begin{definition}
\label{def22}
For a measurable space $(\Omega, \f)$, we define the set of the minimal elements 
in the $\sigma$-algebra $\f$ by
\begin{equation}
\label{atoms}
\mathcal{A}(\f) = \{A\in\f: A\neq \emptyset , \text{ and if } 
F\in\f \text{ and } F\subset A , \text{ then } F=A\} .
\end{equation}
\end{definition}

Eq.~\eqref{atoms} means that $A\in\mathcal{A}(\f)$ can not be `split' in $\f$, which is why
elements of $\mathcal{A}(\f)$ are also referred to as `atoms' of $\mathcal{F}$. 
Obviously, $\mathcal{A}(\f) \subset \f$.
For $|\f|<+\infty$, it is therefore easy to see that $\mathcal{A}(\f)$
is a partition of $\Omega$ and that 
\begin{equation}
\label{generation}
\mathcal{F} = \sigma(\mathcal{A}(\f))
\end{equation}
since any non-empty $F\in\f$ can be written as a finite union of elements in $\mathcal{A}(\f)$.

\begin{definition}
\label{def2}
For $t\in\mathbb{T}\cup\{+\infty\}$, we denote the set of minimal elements in $\f_t$ by
\begin{equation}
\label{atoms1}
\mathcal{A}_t = \mathcal{A}(\f_t) 
= \{A\in\f_t: A\neq \emptyset , \text{ and if } 
F\in\f_t \text{ and } F\subset A , \text{ then } F=A\} .
\end{equation}
Further, we define
\begin{eqnarray}
\label{a1}
\mathcal{A}^{t}_{\tau} & = & 
\{A\in\mathcal{A}_t: A\subset\{\tau=t\}\} \quad (t\in\mathbb{T}\cup\{+\infty\}) ,\\
\label{a3}
\mathcal{A}_{\tau} & = & \bigcup_{t\in\mathbb{T}\cup\{+\infty\}} \mathcal{A}^{t}_{\tau} .
\end{eqnarray}
\end{definition}

Note that the $\mathcal{A}^{t}_{\tau}$ are disjoint for $t\in\mathbb{T}\cup\{+\infty\}$.

\begin{theorem}
\label{prop1}
The elements of $\mathcal{A}_{\tau}$ are minimal elements of $\f_\tau$, 
i.e.~$\mathcal{A}_{\tau} \subset \mathcal{A}(\mathcal{F}_{\tau})$. 
If $|\f_{\infty}|<+\infty$, then $\mathcal{A}_{\tau}$ is the set of all minimal elements of $\f_\tau$,
i.e.~$\mathcal{A}_{\tau} = \mathcal{A}(\mathcal{F}_{\tau})$,
and $\f_\tau = \sigma(\mathcal{A}_{\tau}) = \sigma(\mathcal{A}(\mathcal{F}_{\tau}))$ .
\end{theorem}

\begin{proof}
(i) 
$\mathcal{A}_{\tau} \subset \mathcal{F}_{\tau}$: 
Let $A\in \mathcal{A}_{\tau}$, so, for some $s\in\mathbb{T}\cup\{+\infty\}$, $A\in\mathcal{A}_{s}$
and $A\subset\{\tau = s\}$ and therefore $A \cap \{\tau = s\} = A$. 
Assume now $t<s$ for some $t\in\mathbb{T}$. Then 
$A \cap \{\tau \leq t\} = A \cap \{\tau = s\} \cap \{\tau \leq t\} = \emptyset \in \f_t$.
For $t\in\mathbb{T}$, assume now $t \geq s$. Then 
$A \cap \{\tau \leq t\} = A \cap \{\tau = s\} \cap \{\tau \leq t\}
= A \cap \{\tau = s\} = A \in \f_s \subset \f_t$. 
Hence, $A \cap \{\tau \leq t\} \in \f_t$ for all $t\in\mathbb{T}$, and therefore
$\mathcal{A}_{\tau} \subset \mathcal{F}_{\tau}$.
(ii) 
$A\in\mathcal{A}_{\tau}$, $F\in\mathcal{F}_{\tau}$ and $F\subset A$ implies $F=A$:
(a) 
Assume that $A\in\mathcal{A}^{\infty}_{\tau}$ and $F\in\mathcal{F}_{\tau}$ with $F\subset A$.
Clearly, $A\in\mathcal{A}_{\infty}$, implying $F=A$ since $F\in\mathcal{F}_{\infty}$.
(b) 
Assume $A\in\mathcal{A}^{t}_{\tau}$ for some $t\in\mathbb{T}$
and $F\in\mathcal{F}_{\tau}$ with $F\subset A$. Therefore,
$A\in\mathcal{A}_t$ and $F\subset A\subset\{\tau=t\}$, and hence
$F\cap\{\tau=t\}=F$. As $F\in\f_\tau$, one has $F\cap\{\tau\leq t\}\in\f_t$. Since
$\{\tau=t\}\in\f_t$, $F\cap\{\tau\leq t\}\cap\{\tau = t\} = F\cap\{\tau = t\} = F \in\f_t$,
but $A\in\mathcal{A}_t$, and therefore $F=A$.
(i) and (ii) prove the first statement of the theorem. Assume now $|\f_{\infty}|<+\infty$.
$\mathcal{A}_{\tau}$ is then a partition of $\Omega$, because any two distinct sets in $\mathcal{A}_{\tau}$
are disjoint, and, since $\{\tau=t\}\in\f_t$ for $t\in\mathbb{T}\cup\{+\infty\}$, 
one has $\bigcup \mathcal{A}^t_{\tau} = \{\tau = t\}$, so $\bigcup \mathcal{A}_{\tau} = \Omega$.
This proves the second statement. The third statement follows by Eq.~\eqref{generation}.
\end{proof}

The following result is well known.

\begin{proposition}
\label{prop_3}
For $|\f_{\infty}|<+\infty$, $\sigma(\tau) \subset \f_\tau$ and, in general, $\sigma(\tau) \neq \f_\tau$.
\end{proposition}

\begin{proof}
$\{\tau=t\}$ $(t\in\mathbb{T})$ and $\{\tau=+\infty\}$ are the minimal elements of $\sigma(\tau)$
if they are non-empty, but it is well known and straightforward to see that these sets are elements
of $\f_\tau$, too. Therefore, $\sigma(\tau) \subset \f_\tau$.
It is an easy exercise to find examples where $\sigma(\tau) \neq \f_\tau$ (see example below).
\end{proof}

We can interpret the filtered probability space $(\Omega, \f_\infty, (\f_t)_{t\in\mathbb{T}}, \p)$
as a random `experiment' or `experience' that develops over time.

The common interpretation of the filtration $(\f_t)_{t\in\mathbb{T}}$ is that if it is possible
to repeat the experiment arbitrarily often until time $t\in\mathbb{T}$, the maximum information
that can be obtained about the experiment is $\f_t$. Hence, $\f_t$ represents 
(potentially) available information up to time $t$.

Using Theorem \ref{prop1} for $|\f_{\infty}|<+\infty$, 
we can interpret the stopping time $\sigma$-algebra $\F_\tau$ as the maximum information
that can be obtained from repeatedly carrying out the experiment up to the random time $\tau$.
This is straightforward from the definition of the $\mathcal{A}^{t}_{\tau}$ in \eqref{a1}.
While this interpretation is, of course, generally known, minimal sets are usually not used
to derive it. However, the example below will illustrate that this is a very natural
way of interpreting stopping time $\sigma$-algebras.

For an intuitive interpretation and representation of stopping times see Fischer (2011).\\


{\bf Example.} 
In Figure \ref{fig:example}, we see the usual interpretation of a discrete
time finite space filtration as a stochastic tree. 
In such a setting, the set of paths of maximal length represents $\Omega$.
In this case, $\Omega = \{\omega_1,\ldots,\omega_8\}$.
A path up to some node at time $t$ (here $\mathbb{T}=\{0,1,2,3\}$
and $\f_3 = \f_{\infty} = \mathcal{P}(\Omega)$) 
represents the set of those paths of full length that have this path up to time $t$ in common.
The paths up to time $t$ represent the minimal elements (atoms) $\mathcal{A}_t$ of $\f_t$.
For instance, in the example of Fig.~\ref{fig:example}, 
\begin{equation}
\mathcal{A}_1 = \{\{\omega_1,\omega_2\},\{\omega_3,\omega_4\},\{\omega_5,\ldots,\omega_8\}\} .
\end{equation}
A stopping time $\tau$ (see Fig.~\ref{fig:example}) and the corresponding stopping process
$X^{\tau}$ (see Fischer (2011)) are given ($X^{\tau}$ is given by the values of the nodes of
the tree). The intuitive interpretation of the stopping time as the random time
at which $X^{\tau}$ jumps to (`hits') $0$ becomes clear (see boxed zeros). Furthermore, the paths
up to those boxed zeros represent the minimal elements $\mathcal{A}_\tau$ of the stopping time
$\sigma$-algebra $\F_\tau$. This follows of course from the fact that for any 
$A\in\mathcal{A}^{t}_{\tau}$ one has $A\subset\{\tau=t\}$ by \eqref{a1} and therefore
$X^{\tau}_t(A)=0$, but $X^{\tau}_s(A)=1$ for $s<t$. So, in a tree example such as the given one,
the `frontier of first zeros' (or, preciser, the paths leading up to it)
describes the stopping time $\sigma$-algebra (as well as the stopping time itself).
Therefore, it becomes very obvious in what sense $\f_{\tau}$ contains the information
in the system that can be explored up to time $\tau$. 
In the case of the example,
\begin{equation}
\mathcal{A}_\tau 
= 
\{\{\omega_1,\omega_2\},\{\omega_3\},\{\omega_4\},\{\omega_5,\omega_6\},\{\omega_7\},\{\omega_8\}\} .
\end{equation}
It is also easy to see that in this case the minimal elements of $\sigma(\tau)$,
denoted by $\mathcal{A}(\sigma(\tau))$, are given by
\begin{equation}
\mathcal{A}(\sigma(\tau)) 
= 
\{\{\omega_1,\omega_2\},\{\omega_5,\omega_6\},\{\omega_3,\omega_4,\omega_7,\omega_8\}\} .
\end{equation}
Therefore, $\mathcal{A}_\tau \neq \mathcal{A}(\sigma(\tau))$ and, hence, $\sigma(\tau) \neq \f_\tau$,
as stated in Prop.~\ref{prop_3}.



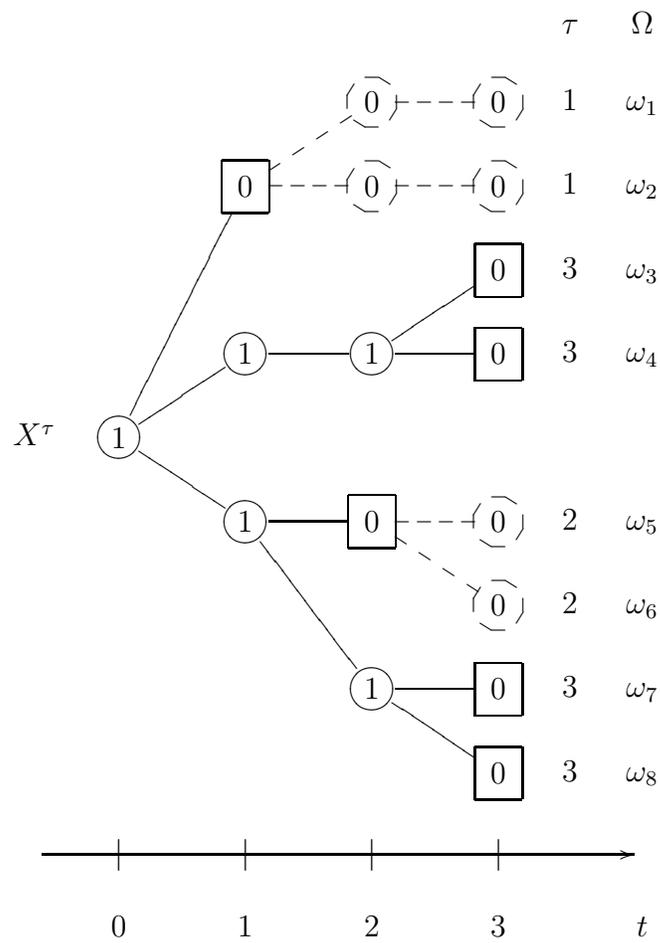
\begin{figure}[b]
$$
\xymatrix @-1pc {
& & & & & & & &\tau& \Omega \\
& & & & &*++[o][F--]{0}\ar@{--}[rr]& &*++[o][F--]{0}&1& \omega_1 \\
& & &*++[F-]{0}\ar@{--}[rr]\ar@{--}[urr]  & &*++[o][F--]{0}\ar@{--}[rr]& &*++[o][F--]{0}&1& \omega_2 \\
& & & & & & &*++[F-]{0}&3& \omega_3 \\
& & &*++[o][F-]{1}\ar@{-}[rr] & &*++[o][F-]{1}\ar@{-}[urr]\ar@{-}[rr]& &*++[F-]{0}&3& \omega_4 \\
X^{\tau} 
&*++[o][F-]{1}\ar@{-}[drr]\ar@{-}[urr]\ar@{-}[uuurr]& & & & & & & & \\
& & &*++[o][F-]{1}\ar@{-}[rr]\ar@{-}[ddrr] & &*++[F-]{0}\ar@{--}[drr]\ar@{--}[rr]& &*++[o][F--]{0}&2& \omega_5 \\
& & & & & & &*++[o][F--]{0}&2& \omega_6 \\
& & & & &*++[o][F-]{1}\ar@{-}[drr]\ar@{-}[rr]& &*++[F-]{0}&3& \omega_7 \\
& & & & & & &*++[F-]{0}&3& \omega_8 \\
\ar[rrrrrrrrr] &|& &|& &|& &|& & \\
&0& &1& &2& &3& &t  \\
}
$$
\caption{\label{fig:example} Example.}
\end{figure}


\end{document}